\newcommand \datum {April 29, 2021}

\documentclass[reqno,oneside]{amsart}

\usepackage{amssymb,latexsym}
\usepackage{amsmath}
\usepackage{amscd}
\usepackage{graphicx}
\usepackage{color}
\usepackage{enumerate}
\numberwithin{equation}{section}
\theoremstyle{plain}
 \newtheorem{theorem}{Theorem}[section]
 \newtheorem{lemma}[theorem]{Lemma}
 \newtheorem{corollary}[theorem]{Corollary}
 
\theoremstyle{definition}
 \newtheorem{observation}[theorem]{Observation}
 \newtheorem{definition}[theorem]{Definition}
 \newtheorem{remark}[theorem]{Remark}
\newenvironment{enumeratei}{\begin{enumerate}[\quad\upshape (i)]} {\end{enumerate}}
\newcommand\mmm {\kern 5pt\wedge \kern 5pt}

\newcommand\blk [1] {\boxed{#1}} %for blocks  
\newcommand \ncm {4.5cm}
\newcommand \ocm {4.8cm}
\newcommand \set [1]{\{#1\}}
\newcommand \tuple [1] {\langle #1 \rangle}
\newcommand \pair [2] {\tuple{#1,#2}}
\newcommand \tbf[1] {\textbf{#1}}  
\newcommand \defiff {\overset{\textup{def}}{\iff}}
\newcommand \Equ[1] {\textup{Equ}(#1)}
\newcommand \Quo[1] {\textup{Quo}(#1)}

\newcommand \Nplu {\mathbb N^+}

\newcommand \vonal {\noalign{\hrule}} %for the table
\newcommand \restrict [2]  {#1\mathord\rceil_{\kern-1.5pt #2} }
\newcommand\red[1]{{\textcolor{red}{#1}}}

\begin{document}
\setlength\arraycolsep{1.5pt} %For "diagrams", i.e, matrices 

\title{(1\,+\,1\,+\,2)-generated  lattices of quasiorders}

\author[D. Ahmed]{Delbrin Ahmed}
\email{delbrin@math.u-szeged.hu; delbrinhahmed@gmail.com}
\address{University of Szeged, Bolyai Institute\, and the University of Duhok}
\urladdr{http://www.math.u-szeged.hu/~delbrin/}

\author[G.\ Cz\'edli]{G\'abor Cz\'edli}
\email{czedli@math.u-szeged.hu}
\urladdr{http://www.math.u-szeged.hu/~czedli/}
\address{University of Szeged, Bolyai Institute. 
%Szeged, Aradi v\'ertan\'uk tere 1, HUNGARY 6720
}

\begin{abstract} A lattice is $(1+1+2)$-generated if it has a four-element generating set such that exactly two of the four generators are comparable.
We prove that the lattice $\Quo n$ of all quasiorders (also known as preorders) of an $n$-element set is $(1+1+2)$-generated for
$n=3$ (trivially), $n=6$ (when $\Quo 6$ consists of $209\,527$ elements), $n=11$, and for every natural number $n\geq 13$. 
In 2017, the second author and J. Kulin proved
that $\Quo n$ is $(1+1+2)$-generated if either $n$ is odd and at least $13$ or $n$ is even and at least $56$. Compared to the 2017 result, this paper presents twenty-four new numbers $n$ such that $\Quo n$ is $(1+1+2)$-generated. Except for $\Quo 6$, an extension of Z\'adori's method is used.
\end{abstract}

\thanks{This research of the second author was supported by the National Research, Development and Innovation Fund of Hungary under funding scheme K 134851.}



\date{\datum\hfill{\red{\tbf{Hint:} check the authors' web pages for possible updates}}}

\maketitle

\section{Introduction}\label{sect:intro} 

\subsection{Outline}\label{ssect:outline}
The present section is introductory, and it is structured as follows. Subsection~\ref{ssect:basics} contains the basic concepts used in the paper. Subsection~\ref{ssect:history} gives a short historical survey. This survey motivates our target, which is  presented in Subsection~\ref{ssect:target}. Subsection~\ref{ssect:jointauthors} is a comment on the joint authorship.
Sections~\ref{sec:quosix} and \ref{sect:bign} prove for some values of $n$ 
that the lattice $\Quo n$ of quasiorders of an $n$-element finite set is $(1+1+2)$-generated. Sections~\ref{sec:quosix}, the longest section, proves this for $n=3$ and $n=6$ while Section~\ref{sect:bign} is devoted to $n=11$ and all $n\geq 13$.
At the end of Section~\ref{sect:bign}, Remark~\ref{rem:newandopen} summarizes which $\Quo n$ 
are known to be $(1+1+2)$-generated and which are not.

\subsection{Basic concepts}\label{ssect:basics}
Given a set $A$, a relation $\rho\subseteq A^2$ is a \emph{quasiorder} (also known as a \emph{preorder}) if $\rho$ is reflexive and transitive. With respect to set inclusion, the set of all quasiorders of $A$ form a lattice $\Quo A=\tuple{\Quo A,\subseteq}$, the \emph{quasiorder lattice} of $A$. The meet of $\rho,\tau\in \Quo A$ is their intersection, and so we can write that $\rho\wedge\tau=\rho\cap\tau$. The join $\rho\vee \tau$ of $\rho$ and $\tau$
is the transitive closure of $\rho\cup\tau$. That is, for $x,y\in A$, we have $\pair x y\in \rho\vee \tau$ if and only if there exists an $n\in\Nplu:=\set{1,2,3,4,\dots}$ and there are
elements $z_0=x, z_1, z_2,\dots, z_{n-1}, z_n=y$ in $A$ such that $\tuple{z_{i-1},z_i}\in\rho\cup\tau$ for all $i\in\set{1,\dots,n}$. Symmetric quasiorders are \emph{equivalences} (also known as  \emph{equivalence relations}). The equivalences of $A$ also form a lattice, the \emph{equivalence lattice} $\Equ A$ of $A$, which is a sublattice of $\Quo A$. 

Since we are only interested in these lattices up to isomorphism, we will often write $\Equ{|A|}$ and $\Quo{|A|}$ instead of $\Equ{A}$ and $\Quo{A}$, respectively. In particular, $\Quo 6$, which plays a distinguished role in this paper, is the quasiorder lattice with a six-element underlying set. 
Note that   $\Equ A$ and $\Quo A$ are complete lattices but the concept of \emph{complete} lattices only occurs here in the introductory section when the literature is surveyed. In the rest of the sections, we only deal with \emph{finite} lattices, which are automatically complete.  

A four-element subset $X$ of a poset (partially ordered set) $Y$ is a \emph{$(1+1+2)$-subset} of $Y$ if exactly two elements of $X$ are comparable. A subset $X$ of a lattice $L$ is a \emph{$(1+1+2)$-generating set} of $L$ if  $X$ is a $(1+1+2)$-subset of $L$ that generates $L$. If a lattice $L$ has a $(1+1+2)$-generating set, then we say that $L$ is \emph{$(1+1+2)$-generated}.

\subsection{Earlier results}\label{ssect:history}
In 1976, Poguntke and Rival~\cite{poguntkerival} proved that each lattice can be embedded into a four-generated finite simple lattice. (It turned out much later that three generators are sufficient if we drop simplicity; see Cz\'edli~\cite{czg3embed}.) Partition lattices, which are the same as equivalence lattices up to isomorphism, are well known to be simple. Thus, Pudl\'ak and T\r uma's result that every finite lattice is embeddable into a finite partition lattice, see
~\cite{pudlaktuma}, superseded  Poguntke and Rival's result in 1980. However, Poguntke and Rival's result still served well as the motivation for Strietz~\cite{strietz75} and \cite{strietz77} to prove that $\Equ n$ is four-generated for $3\leq n\in\Nplu$ and it is $(1+1+2)$-generated for $10\leq n\in\Nplu$.

In 1983, Z\'adori~\cite{zadori} gave an entirely new method to find four-element generating sets of $\Equ n$ and extended Strietz's result by proving that  $\Equ n$ is $(1+1+2)$-generated even for $7\leq n\in\Nplu$. His method was the basis of all the more involved methods that were used to find small generating sets of $\Equ A$ and $\Quo A$ in the last three and a half decades. During this period, four-element generating sets
and even $(1+1+2)$-generating sets (in the sense of complete generation) of $\Equ A$ were given for all infinite sets $A$ with ``accessible'' cardinalities; see Cz\'edli~\cite{czgAtEqSmall},
\cite{czgFgenLargeASM}, and \cite{czgEEKgenEqu}.  Even the lion's share of Cz\'edli and Oluoch~\cite{czgoluoch} is based on Z\'adori's method. 
Also, this period witnessed that extensions of his method were used to find small generating sets of $\Quo A$ 
by Chajda and Cz\'edli~\cite{ChCzG}, Cz\'edli \cite{czgFGQL2017}, Cz\'edli and Kulin~\cite{CzGKJconcise}, and Tak\'ach~\cite{takach}, and even the methods used by Dolgos~\cite{dolgos} and Kulin~\cite{KulinQ5} show lots of similarity with Z\'adori's method. Theorem~\ref{thm:czgkj} of this subsection will summarize the strongest results on $(1+1+2)$-generating sets of quasiorder lattices that were proved before the present paper.

Note that sometimes the structure $\tuple{\Quo A, \bigvee,\bigwedge, {}^{-1}}$ 
with $\rho^{-1}:=\set{\pair x y: \pair y x\in\rho}$
rather than the complete lattice $\tuple{\Quo A, \bigvee,\bigwedge}$ 
was considered.  So the title of Tak\'ach~\cite{takach} should not mislead the reader since, after removing the operation $\rho\mapsto \rho^{-1}$, Tak\'ach~\cite{takach}  yields a six-element generating set of $\tuple{\Quo A, \bigvee,\bigwedge}$.

\begin{table}
%\begin{center}
\[
%\lower  0.8 cm
\vbox{\tabskip=0pt\offinterlineskip
\halign{\strut#&\vrule#\tabskip=4pt plus 2pt&
#\hfill& \vrule\vrule\vrule#&
\hfill#&\vrule#&
\hfill#&\vrule#&
\hfill#&\vrule#&
\hfill#&\vrule#&
\hfill#&\vrule#&
\hfill#&\vrule#&
%\hfill#&\vrule#&
%\hfill#&\vrule\tabskip=0.0pt#&
\hfill#&\vrule\tabskip=0.1pt#&
#\hfill\vrule\vrule\cr
\vonal\vonal\vonal\vonal
&&\hfill$n$&&$\,1$&&$\,2$&&$\,3$&&$\,4$&&$\,5$&&$\,6$&&$\,7$&
\cr
\vonal\vonal
&&\hfill$|\Equ n|$&&$1$&&$2$&&$5$&&$15$&&$52$&&$203$&&$877$&
\cr\vonal\vonal
&&\hfill$|\Quo n|$&&$1$&&$4$&&$29$&&$355$&&$6\,942$&&$209\,527$&&$9\,535\,241$&
\cr\vonal\vonal\vonal\vonal
}} 
%\label{cimke}
\]
\caption{$|\Equ n|$ and $|\Quo n|$  for $n\in\set{1,2,\dots,7}$}\label{tableEQ}
\end{table}

In 1983, Z\'adori~\cite{zadori} left  the problem whether $\Equ n$ is $(1+1+2)$-generated for $n\in\set{5,6}$ open. The difficulty of these small values lies in the fact that his method does not work for small $n$. This explains that it took 37 years to solve Z\'adori's problem on $\Equ 5$ and $\Equ 6$; see Cz\'edli and Oluoch~\cite{czgoluoch} for the solution.  While \cite{czgoluoch} contains a traditional proof that $\Equ 6$ is $(1+1+2)$-generated, computer programs were used to show that $\Equ 5$ is not. This shows that ``small'' equivalence lattices create more difficulty than larger ones.  The quotient marks indicate that $\Equ 5$ and $\Equ 6$ are not so small; see Table~\ref{tableEQ}. This table has been taken from Sloane~\cite{sloane}. Note that the $|\Equ n|$ row and the first five numbers of the $|\Quo n|$ row of  Table~\ref{tableEQ} also occur in Chajda and Cz\'edli~\cite{ChCzG} and were obtained by a straightforward computer program twenty-five years ago.

Our knowledge on small generating sets of quasiorder lattices evolved parallel to and in interactions with the analogous question about equivalence lattices. Small generating sets of \emph{infinite} (complete) quasiorder lattices were given in Chajda and Cz\'edli \cite{ChCzG} even before dealing with infinite equivalence lattices.
Surprisingly, it was quasiorders that showed the way how to pass from finite equivalence lattices to infinite ones.  
Prior to the present paper, our knowledge on small generating sets of $\Quo A$ was summarized  in the last sentence of Theorem 1.1 in  Cz\'edli~\cite{czgFGQL2017} and  in Theorem 3.5 and Lemma 3.3 of Cz\'edli and Kulin~\cite{CzGKJconcise} as follows. 

\begin{theorem} [Cz\'edli~\cite{czgFGQL2017} and Cz\'edli and Kulin~\cite{CzGKJconcise}]\label{thm:czgkj} If $A$ is a non-singleton set with accessible cardinality, then the following assertions hold.
\begin{enumeratei}
\item\label{thm:czgkja} If $|A|\neq 4$, then $\Quo A$ is four-generated as a complete lattice.
\item\label{thm:czgkjb} If $13\leq |A|$ is a finite \emph{odd} number, then  $\Quo A$ is $(1+1+2)$-generated.
\item\label{thm:czgkjc} If $56\leq |A|$ is a finite \emph{even} number, then  $\Quo A$ is $(1+1+2)$-generated.
\item\label{thm:czgkjd} If $A$ is infinite, then  $\Quo A$ is $(1+1+2)$-generated as a complete lattice.
\item\label{thm:czgkje} If $A$ is finite and $|A|\geq 3$, then $\Quo A$ is not a three-generated lattice. 
\end{enumeratei}
\end{theorem}

Note  that  ZFC has a model in which all infinite sets are of accessible cardinalities. 
We do not know whether $\Quo A$ has a four-element generating set if $|A|=4$. (Based on our experience with computer programs for equivalence lattices, see Cz\'edli and Oluoch~\cite{czgoluoch}, we guess that this question could be solved by a computer program that would require days of computer time if the same personal computer was used as in case of \cite{czgoluoch}. Developing such a computer program was not pursued at the time of writing.)

Although this paper presents $(1+1+2)$-generating sets for several new values of $n$, the existence of $(1+1+2)$-generating sets remains an unsolved problem for a few values. The only value of $n\geq 2$ for which we know that $\Quo n$ is \emph{not} $(1+1+2)$-generated is $n=2$; this follows trivially from $|\Quo 2|=4$ since a four-element lattice cannot have a three-element antichain.

To conclude this subsection, we note the following about quasiorder lattices.
Armed with our tools based on Z\'adori's method,
the large values of $n$ create less difficulty than its small values (apart from very small values where the problem is trivial or easy). In view of the history of equivalence lattices, it is not a surprise that the present 
paper extends the scope of \eqref{thm:czgkjb} and \eqref{thm:czgkjc} of Theorem~\ref{thm:czgkj} by adding some \emph{slightly smaller} numbers $n$. The surprise is that now we also add a significantly smaller number, $n=6$, where $\Quo 6$ is a huge lattice but Z\'adori's method cannot be used.

\subsection{Target}\label{ssect:target}
We are going to prove that the $209\,527$-element lattice $\Quo 6$ is $(1+1+2)$-generated; see Theorem~\ref{thm:wwQhTszNsJk}, which is the main result of this paper.
We will also prove that $\Quo A$ is $(1+1+2)$-generated for 
$|A|\in\set{11}\cup\set{n\in \Nplu: n\geq 13}$; see Theorem \ref{thm:sBbszmK}.

\subsection{Joint authorship}\label{ssect:jointauthors}
Sections~\ref{sect:intro} and \ref{sec:quosix} are joint work  of the two authors.
The contribution of the first author to Section~\ref{sec:quosix} is about sixty percent. Section~\ref{sect:bign} is due to the second author.

\section{A $(1+1+2)$-generating set of $\Quo 6$}\label{sec:quosix}

The least quasiorder $\set{\pair x x: x\in A}$ of $A$ will be denoted by $\Delta=\Delta_A$.
For elements $x$ and $y$ of $A$, the following two members of $\Quo A$ will play a particularly important role in our proofs:
\begin{equation}
q(x,y)=\set{\pair x y}\cup\Delta\,\,\text{ and }\,\,e(x,y)=e(y,x)=\set{\pair x y, \pair y x}\cup\Delta.
\label{eq:wsZlshRnGpcngRszq}
\end{equation}
We allow that $x=y$; however, $q(x,x)=\Delta$ and $e(x,x)=\Delta$ will not play any significant role in our proofs.  The atoms of $\Quo A$ and those of $\Equ A$ are exactly the $q(x,y)$ and the $e(x,y)$ with $x\neq y\in A$. 
The importance of $q(x,y)$ and $e(x,y)$ lies in the following well-known and trivial fact: for any non-singleton set $A$ and for every $\rho\in \Quo A$ and $\theta\in\Equ A$, 
\begin{equation}
\rho=\bigvee\set{q(x,y): \pair x y\in\rho}\quad\text{ and }\quad
\theta=\bigvee\set{e(x,y): \pair x y\in\theta}.
\label{dg:eq:wbhWrnGtrSg}
\end{equation}

\begin{figure}[ht]
\centerline
{\includegraphics[scale=1]{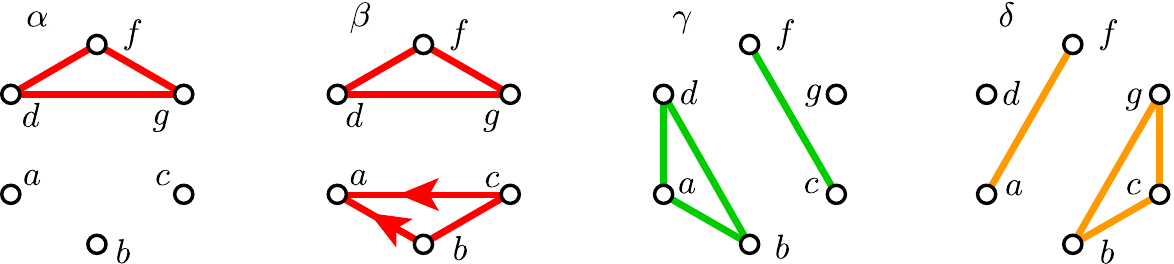}}
\caption{$\alpha$, $\beta$, $\gamma$, and $\delta$}\label{figabgd}
\end{figure}

Next, let $A=\set{a,b,c,d,f,g}$. Note the gap in the alphabetical order; $e$ is not in $A$ since $e$ is used to denote an atom of $\Equ A$ in \eqref{eq:wsZlshRnGpcngRszq} and many times later. We define the following quasiorders of $A$:
\begin{equation}
\begin{aligned}
\alpha&:=e(d,f)\vee e(f,g), && \beta:=\alpha\vee e(b,c)\vee q(b,a)\cr
\gamma&:=e(a,b)\vee e(a,d)\vee e(c,f),&&\delta:=e(b,c)\vee e(c,g)\vee e(a,f).
\end{aligned}
\label{eq:nGymitDlBsGgRt}
\end{equation}
These quasiorders are visualized in Figure~\ref{figabgd} by the corresponding directed graphs. (The edges without arrows are directed in both ways.)
 Namely, for $\rho\in\set{\alpha,\beta,\gamma,\delta}$ and $x,y\in A$, we have that $\pair x y\in\rho$ if and only if there is a directed path (possibly of length 0) in the graph  corresponding to $\rho$ in Figure~\ref{figabgd}.  

For the sake of the following remark, let
$\tuple{u_1,u_2,u_3,u_4,u_5,u_6}:=\tuple{b,a,c,d,f,g}$ and $\beta^\star:=\beta\vee q(a,b)=\beta\vee q(u_2,u_1)$.
\begin{remark}\label{rem:whGLdD}
Each of $\alpha$, $\beta^\star$, $\gamma$, and $\delta$ is an equivalence, $\alpha<\beta^\star$, and it has been proved in Cz\'edli and Oluoch~\cite{czgoluoch} that $\set{\alpha,\beta^\star,\gamma,\delta}$ is a $(1+1+2)$-generating set of  $\Equ 6=\Equ{\set{u_1,u_2,u_3,u_4,u_5,u_6}}$. 
\end{remark}
We are going to prove the following statement.
\begin{theorem}\label{thm:wwQhTszNsJk} 
With the quasiorders defined in \eqref{eq:nGymitDlBsGgRt},  $\set{\alpha,\beta,\gamma,\delta}$ is a $(1+1+2)$-generating set of the quasiorder lattice $\Quo 6=\Quo{\set{a,b,c,d,f,g}}$. Hence, $\Quo 6$ is  $(1+1+2)$-generated.
\end{theorem}

As Remark~\ref{rem:whGLdD} indicates, our generating set is only slightly different from the one used in Cz\'edli and Oluoch~\cite{czgoluoch}. However, this little difference results in a substantial change in the complexity of the proofs. Indeed, while only six equations were necessary in \cite{czgoluoch} to prove that $\set{\alpha,\beta^\star,\gamma,\delta}$ generates $\Equ 6$, we are going to use  twenty-five equations,  \eqref{gd:eq:BrNaa}--\eqref{gd:eq:BrNcb}, to prove Theorem~\ref{thm:wwQhTszNsJk}.

\begin{proof}[Proof of Theorem~\ref{thm:wwQhTszNsJk}]
First, we fix our notation and describe the corresponding technique. 
For $\rho\in\Quo A$, let $\Theta(\rho):=\rho\cap\rho^{-1}=\set{\pair x y: \pair x y\in \rho\text{ and }\pair y x\in\rho}$, which is the largest equivalence relation of $A$ included in $\rho$. On the quotient set $A/\Theta(\rho)$, we can define a relation  $\rho/\Theta(\rho)$ as follows: for $\Theta(\rho)$-blocks $x/\Theta(\rho)$ and 
 $y/\Theta(\rho)$ in  $A/\Theta(\rho)$, we let
\begin{equation*}
\tuple{x/\Theta(\rho), y/\Theta(\rho)} \in \rho/\Theta(\rho) \defiff \pair x y \in \rho.
\end{equation*}
It is a well-known folkloric fact that $\rho/\Theta(\rho)$ is well defined and it is a partial order, the so-called \emph{order induced by} $\rho$. Hence, $A/\Theta(\rho)= \tuple{ A/\Theta(\rho),\rho/\Theta(\rho) }$ is a poset. For several choices of $\rho$, 
we will frequently draw the Hasse diagram of this poset in order to give a visual description of $\rho$. In such a diagram, the $\Theta(\rho)$-blocks are indicated by rectangles. However, we adopt the following convention:
\begin{equation}
\parbox{8.3cm}{if $\set{x}$ is a singleton block of $\Theta(\rho)$ such that for every $\set x\neq Y\in A/\Theta(\rho)$ we have that $\tuple{\set x,Y}\notin \rho/\Theta(\rho)$ and $\tuple{Y,\set x}\notin \rho/\Theta(\rho)$, then $\set x$ is not indicated in the Hasse diagram of $A/\Theta(\rho)$.}
\label{gd:pbx:swWsTrhPJzS}
\end{equation}
In other words, if $x\in A$ has the property that 
\[(\forall y\in A)\,\bigl(\set{\pair x y,\pair y x}\cap\rho\neq\emptyset \Longrightarrow x=y  \bigr),
\] 
then (the necessarily singleton) block $x/\Theta(\rho)$ is not indicated in the Hasse diagram.
For example, the quasiorders defined in \eqref{eq:nGymitDlBsGgRt} are visualized by diagrams as follows.
\begin{equation}
\alpha:\blk{d,f,g},\text{ }
\beta: 
\begin{matrix}
\blk{a\phantom{f}}&\blk{d,f,g}\cr
|&\cr
\blk{b,c}
\end{matrix},\text{ }
\gamma:
\begin{matrix}
\blk{a,b,d}&\blk{c,f}
\end{matrix}
,\text{ }
\delta: 
\begin{matrix}
\blk{b,c,g}&\blk{a,f}\,.
\end{matrix}
\label{gd:eq:FgWRmkMjs}
\end{equation}
Since we are going to perform a lot of computations with quasiorders, convention \eqref{gd:pbx:swWsTrhPJzS} and the above-mentioned visual approach will be helpful 
for the reader in the rest of the proof. 
Note that if a diagram according to our convention is given and $x\neq y\in A$, then we have $\pair x y\in \rho$ if and only if both the block $x/\Theta(\rho)$ of $x$ and that of $y$ are drawn in the (Hasse) diagram and   $x/\Theta(\rho) \leq y/\Theta(\rho)$ according to the  diagram. In particular,  if $x$ and $y$ are in the same block, then $\pair x y\in \rho$. Note also that our computations in the proof never require dealing with pairs of the form $\pair x x$. Although the following observation is quite easy to prove, it will substantially ease our task.

\begin{observation}[Disjoint Paths Principle]\label{obs:DPPsTrs} For $k,s\in \Nplu$ and a set $B$, let $x,y,\,\,u_0=x, u_1,\dots, u_{k-1}, u_k=y,\,\,v_0=x, v_1,\dots, v_{s-1}, v_s=y $ be elements
of $B$ such that $\set{u_1,\dots, u_{k-1}}\cap\set{ v_1,\dots, v_{s-1}}=\emptyset$, $|\set{u_1,\dots, u_{k-1}}|=k-1$, and $|\set{v_1,\dots, v_{s-1}}|=s-1$.
For $i\in\set{1,\dots,k}$ and $j\in\set{1,\dots,s}$, let $p_i\in \set{e,q}$ and $r_j\in\set{e,q}$; see \eqref{eq:wsZlshRnGpcngRszq} for the meaning of $q$ and $e$.
Assume that there is an  $i'\in\set{1,\dots,k}$ such that $p_{i'}=q$ or there is a $j'\in\set{1,\dots,s}$ such that $r_{j'}=q$. Then
\begin{equation}
q(x,y) =\Bigl( \bigvee_{i=1}^k p(u_{i-1},u_i)\Bigr)\wedge
\Bigl(\bigvee_{j=1}^s r(v_{j-1},v_j)\Bigr).
\label{eq:wzBmnGVshszTTRkTv}
\end{equation}
\end{observation}

Similar observations (sometimes under the name ``Circle Principle'') have previously been formulated in 
Cz\'edli~\cite{czgAtEqSmall}, \cite[Lemma 2.1]{czgFGQL2017},
Cz\'edli and Kulin~\cite[Lemma 2.5]{czgEEKgenEqu},
Kulin~\cite[Lemma 2.2]{KulinQ5},  
and have been used implicitly in 
Chajda and Cz\'edli~\cite{ChCzG},
Cz\'edli~\cite{czgFgenLargeASM} and \cite{czgEEKgenEqu},
and Tak\'ach~\cite{takach}. However, Observation~\ref{obs:DPPsTrs} is slightly stronger than its precursors.

The argument proving Observation~\ref{obs:DPPsTrs} runs as follows. We can assume that $x\neq y$.
Let $\rho$ denote the quasiorder given on the right of the equality sign in \eqref{eq:wzBmnGVshszTTRkTv}. 
Since the pair $\pair x y$ belongs to both meetands in \eqref{eq:wzBmnGVshszTTRkTv} by transitivity, the inequality $q(x,y)\leq \rho$ is clear. Since $\set{u_1,\dots, u_{k-1}}\cap\set{ v_1,\dots, v_{s-1}}=\emptyset$, we obtain that 
\begin{equation}
\rho\leq \Bigl( \bigvee_{i=1}^k e(u_{i-1},u_i)\Bigr)\wedge
\Bigl(\bigvee_{j=1}^s e(v_{j-1},v_j)\Bigr)=e(x,y).
\label{eq:sJflSbbTbfgrRSl}
\end{equation}
Using that the existence of $i'$ or $j'$ 
together with $|\set{u_1,\dots, u_{k-1}}|=k-1$, and $|\set{v_1,\dots, v_{s-1}}|=s-1$
easily exclude that $\pair y x\in \rho$, \eqref{eq:sJflSbbTbfgrRSl} implies that $\rho\leq q(x,y)$. Combining this with the previously established converse inequality, we conclude \eqref{eq:wzBmnGVshszTTRkTv} and the validity of Observation~\ref{obs:DPPsTrs}. Note that, for brevity, we will often reference \eqref{eq:wzBmnGVshszTTRkTv} rather than Observation~\ref{obs:DPPsTrs}.
 
Next, resuming the proof of Theorem~\ref{thm:wwQhTszNsJk}, 
let $S$ denote the sublattice generated by $\set{\alpha,\beta,\gamma,\delta}$ in $\Quo 6=\Quo{\set{a,b,c,d,f,g}}$.
Since $S$ is closed with respect to $\wedge$ and $\vee$, it will be clear that the quasiorders given on the left of the equality signs below in \eqref{gd:eq:BrNaa}--\eqref{gd:eq:BrNcb} all belong to $S$, provided we use quasiorders already in $S$ on the right of our equality signs. (To see that the quasiorders on the right are in $S$, we will mention the relevant earlier equations except, possibly, \eqref{gd:eq:FgWRmkMjs}.)
We also need to check that the equalities we claim below hold in $\Quo 6$; we are going to check this either with the help of the diagrams given on the right of the equalities in question, or this will prompt follow from \eqref{eq:wzBmnGVshszTTRkTv}. After these instructions how to read the displayed equations below, we are ready to compute; the details are easy to follow provided \eqref{gd:eq:FgWRmkMjs} is in the reader's visual field.

%\vskip -12pt

%\setlength{\jot}{3pt} = default!
\setlength{\jot}{10pt}
\allowdisplaybreaks
{\begin{align}
&e(b,c)=\beta\wedge\delta\text{ by \eqref{gd:eq:FgWRmkMjs};}
&&
\begin{matrix}
\blk{a\phantom{f}}&\blk{d,f,g}\cr
|&\cr
\blk{b,c}
\end{matrix} 
\mmm 
\begin{matrix}
\blk{b,c,g}&\blk{a,f}
\end{matrix}\,.
      \label{gd:eq:BrNaa}\\
&q(b,a)=\beta\wedge\gamma\text{ by \eqref{gd:eq:FgWRmkMjs};}
&&
\begin{matrix}
\blk{a\phantom{f}}&\blk{d,f,g}\cr
|&\cr
\blk{b,c}
\end{matrix} 
\mmm 
\begin{matrix}
\blk{a,b,d}&\blk{c,f}
\end{matrix}\,.
       \label{gd:eq:BrNab}\\ 
&\parbox{\ncm}{$e(d,f)=\alpha\wedge(\gamma\vee e(b,c))$ 
by \eqref{gd:eq:FgWRmkMjs} and \eqref{gd:eq:BrNaa};}
&&
\blk{d,f,g} \mmm \blk{a,b,d,c,f} \,.
       \label{gd:eq:BrNac}\\ 
&\parbox{\ncm}{$q(g,f)=\alpha\wedge(\delta\vee q(b,a))$ 
by \eqref{gd:eq:FgWRmkMjs} and \eqref{gd:eq:BrNab};}
&&
\blk{d,f,g} \mmm 
\begin{matrix}
\blk{a,f}\cr
|\cr
\blk{b,c,g}
\end{matrix} \,.
      \label{gd:eq:BrNad}\\ 
&\parbox{\ncm}{$e(a,d)=\gamma\wedge(e(d,f)\vee \delta)$ 
by \eqref{gd:eq:FgWRmkMjs} and \eqref{gd:eq:BrNac};}
&&
\begin{matrix} \blk{a,b,d}&\blk{c,f} \end{matrix}
\mmm \begin{matrix} \blk{b,c,g}&\blk{a,f,d} \end{matrix} \,.
      \label{gd:eq:BrNae}\\ 
&\parbox{\ncm}{$q(g,c)=\delta\wedge(q(g,f)\vee \gamma)$ 
by \eqref{gd:eq:FgWRmkMjs} and \eqref{gd:eq:BrNad};}
&&
\begin{matrix} \blk{b,c,g}&\blk{a,f} \end{matrix} \mmm
\begin{matrix}
\blk{a,b,d}&\blk{c,f}\cr
&|\cr
&\blk{g}
\end{matrix}\,.
      \label{gd:eq:BrNaf}\\ 
&\parbox{\ncm}{$e(a,f)=\delta\wedge(e(a,d) \vee e(d,f))$ 
by \eqref{gd:eq:FgWRmkMjs}, \eqref{gd:eq:BrNae}, and \eqref{gd:eq:BrNac};}
&&
\begin{matrix} \blk{b,c,g}&\blk{a,f} \end{matrix} \mmm
\blk{a,d,f}\,.
      \label{gd:eq:BrNag}\\ 
&\parbox{\ncm}{$q(g,a)=(q(g,f)\vee e(f,a))\wedge(q(g,c)\vee e(c,b)\vee q(b,a))$}
&&
\parbox{\ocm}{by \eqref{eq:wzBmnGVshszTTRkTv}, \eqref{gd:eq:BrNad},  \eqref{gd:eq:BrNag},  \eqref{gd:eq:BrNaf},  \eqref{gd:eq:BrNaa}, and \eqref{gd:eq:BrNab}.}
      \label{gd:eq:BrNah}\\ 
&\parbox{\ncm}{$q(g,d)=(q(g,f)\vee e(f,d))\wedge(q(g,a)\vee e(a,d))$}
&&
\parbox{\ocm}{by  \eqref{eq:wzBmnGVshszTTRkTv},  \eqref{gd:eq:BrNad}, \eqref{gd:eq:BrNac}, \eqref{gd:eq:BrNah}, and \eqref{gd:eq:BrNae}.}
      \label{gd:eq:BrNai}\\ 
&\parbox{\ncm}{$q(b,d)=(q(b,a)\vee e(a,d))\wedge(\delta\vee q(g,d))$ 
by \eqref{gd:eq:BrNab}, \eqref{gd:eq:BrNae}, and \eqref{gd:eq:BrNai};}
&&
\begin{matrix} \blk{a,d}\cr |\cr \blk{b} \end{matrix} \mmm
\begin{matrix} \blk{d}\cr |\cr \blk{b,c,g}&\blk{a,f} \end{matrix}\,.
      \label{gd:eq:BrNba}\\ 
&\parbox{\ncm}{$q(g,b)=(q(g,c)\vee e(c,b))\wedge(q(g,d)\vee\gamma)$ 
by \eqref{gd:eq:BrNaf}, \eqref{gd:eq:BrNaa}, and \eqref{gd:eq:BrNai};}
&&
\begin{matrix}\blk{c,b}\cr | \cr\blk{g}\end{matrix}\mmm
\begin{matrix} \blk{a,b,d}&\blk{c,f}\cr |\cr \blk{g} \end{matrix}\,.
      \label{gd:eq:BrNbb}\\ 
&\parbox{\ncm}{$q(b,f)=(q(b,a)\vee e(a,f))\wedge(q(b,d)\vee e(d,f))$}
&&
\parbox{\ocm}{by  \eqref{eq:wzBmnGVshszTTRkTv},  \eqref{gd:eq:BrNab}, \eqref{gd:eq:BrNag}, \eqref{gd:eq:BrNba}, and \eqref{gd:eq:BrNac}.}
      \label{gd:eq:BrNbc}\\ 
&\parbox{\ncm}{$q(c,f)=(e(c,b)\vee q(b,f))\wedge \gamma$ 
by \eqref{gd:eq:BrNaa} and \eqref{gd:eq:BrNbc};}
&&
\begin{matrix}\blk{f}\cr | \cr\blk{c,b}\end{matrix} \mmm 
\begin{matrix} \blk{a,b,d}&\blk{c,f} \end{matrix}\,.
      \label{gd:eq:BrNbd}\\ 
&\parbox{\ncm}{$q(b,c)=e(b,c)\wedge(q(b,f)\vee \gamma)$ by 
\eqref{gd:eq:BrNaa} and \eqref{gd:eq:BrNbc};}
&&
\blk{b,c}\mmm
\begin{matrix}\blk{c,f}\cr | \cr\blk{a,b,d}\end{matrix}\,.
      \label{gd:eq:BrNbe}\\ 
&\parbox{\ncm}{$q(d,f)=e(d,f)\wedge(q(b,f)\vee\gamma) $ by 
\eqref{gd:eq:BrNac} and \eqref{gd:eq:BrNbc};}
&&
\blk{d,f}\mmm
\begin{matrix}\blk{c,f}\cr | \cr\blk{a,b,d}\end{matrix}\,.
      \label{gd:eq:BrNbf}\\ 
&\parbox{\ncm}{$q(a,f)=e(a,f)\wedge(q(b,f)\vee\gamma) $ by 
\eqref{gd:eq:BrNag} and \eqref{gd:eq:BrNbc};}
&&
\blk{a,f}\mmm
\begin{matrix}\blk{c,f}\cr | \cr\blk{a,b,d}\end{matrix}\,.
      \label{gd:eq:BrNbg}\\ 
&\parbox{\ncm}{$q(d,a)=e(d,a)\wedge(q(d,f)\vee e(f,a))$} 
&&
\parbox{\ocm}{by \eqref{eq:wzBmnGVshszTTRkTv},   \eqref{gd:eq:BrNae},  \eqref{gd:eq:BrNbf}, and \eqref{gd:eq:BrNag}.}
      \label{gd:eq:BrNbh}\\ 
&\parbox{\ncm}{$q(f,a)=e(f,a)\wedge(e(f,d)\vee q(d,a))$} 
&&
\parbox{\ocm}{by \eqref{eq:wzBmnGVshszTTRkTv},   \eqref{gd:eq:BrNag}, \eqref{gd:eq:BrNac}, and \eqref{gd:eq:BrNbh}.}
      \label{gd:eq:BrNbi}\\ 
&\parbox{\ncm}{$q(b,g)=(q(b,f)\vee\alpha)\wedge \delta$ by 
\eqref{gd:eq:BrNbc};}
&&
\begin{matrix}\blk{d,f,g}\cr | \cr\blk{b}\end{matrix}
\mmm
\begin{matrix} \blk{b,c,g}&\blk{a,f} \end{matrix}\,.
      \label{gd:eq:BrNbj}\\ 
&\parbox{\ncm}{$q(a,d)=e(a,d)\wedge(q(a,f)\vee e(f,d))$} 
&&
\parbox{\ocm}{by \eqref{eq:wzBmnGVshszTTRkTv},   \eqref{gd:eq:BrNae},  \eqref{gd:eq:BrNbg}, and \eqref{gd:eq:BrNac}.}
      \label{gd:eq:BrNbk}\\ 
&\parbox{\ncm}{$q(f,d)=e(f,d) \wedge(q(f,a)\vee q(a,d))$} 
&&
\parbox{\ocm}{by \eqref{eq:wzBmnGVshszTTRkTv},   \eqref{gd:eq:BrNac},  \eqref{gd:eq:BrNbi}, and \eqref{gd:eq:BrNbk}.}
      \label{gd:eq:BrNbl}\\ 
&\parbox{\ncm}{$q(c,g)=(e(c,b)\vee q(b,g))\wedge(q(c,f)\vee\alpha )$ by 
\eqref{gd:eq:BrNaa}, \eqref{gd:eq:BrNbj}, and \eqref{gd:eq:BrNbd};}
&&
\begin{matrix}\blk{g}\cr | \cr\blk{c,b}\end{matrix} \mmm
\begin{matrix}\blk{d,f,g}\cr | \cr\blk{c}\end{matrix} \,.
      \label{gd:eq:BrNbm}\\ 
&\parbox{\ncm}{$q(c,b)=(q(c,g)\vee q(g,b))\wedge e(c,b)$} 
&&
\parbox{\ocm}{by \eqref{eq:wzBmnGVshszTTRkTv},   \eqref{gd:eq:BrNbm},  \eqref{gd:eq:BrNbb}, and  \eqref{gd:eq:BrNaa}. }
      \label{gd:eq:BrNbn}\\ 
&\parbox{\ncm}{$q(f,g)=\alpha\wedge(\gamma\vee q(c,g))$ by 
\eqref{gd:eq:BrNbm};}
&&
\blk{d,f,g} \mmm
\begin{matrix}&\blk{g}\cr &| \cr\blk{a,b,d}&\blk{c,f}\end{matrix}\,.
      \label{gd:eq:BrNca}\\ 
&\parbox{\ncm}{$q(a,b)=(q(a,f)\vee q(f,g)\vee q(g,b))\wedge\gamma$ by 
\eqref{gd:eq:BrNbg}, \eqref{gd:eq:BrNca}, and \eqref{gd:eq:BrNbb};}
&&
\begin{matrix}\blk{b}\cr | \cr\blk{g}\cr | \cr\blk{f}\cr | \cr\blk{a}\end{matrix} \mmm 
\begin{matrix} \blk{a,b,d}&\blk{c,f} \end{matrix}\,.
      \label{gd:eq:BrNcb}
\end{align}%
}%

In the rest of the proof, we only need the twelve atoms of $\Quo 6=\Quo A$ that are indicated in  Figure~\ref{figbfcircle}. 
Using that these twelve atoms belong to $S$, we conclude  by \eqref{eq:wzBmnGVshszTTRkTv} that, for 
all $x,y\in A$, the quasiorder $q(x,y)$ belongs to $S$ as well. Hence, it follows from \eqref{dg:eq:wbhWrnGtrSg} that $\rho\in S$ for all $\rho\in \Quo A$. Consequently, $S=\Quo A$ and $\set{\alpha,\beta,\gamma,\delta}$ is a generating set of $\Quo A$. Since $\set{\alpha,\beta,\gamma,\delta}$  is a $(1+1+2)$-subset of $\Quo A$, the proof of Theorem~\ref{thm:wwQhTszNsJk} is complete.
\end{proof}

\begin{figure}[ht]
\centerline
{\includegraphics[scale=1]{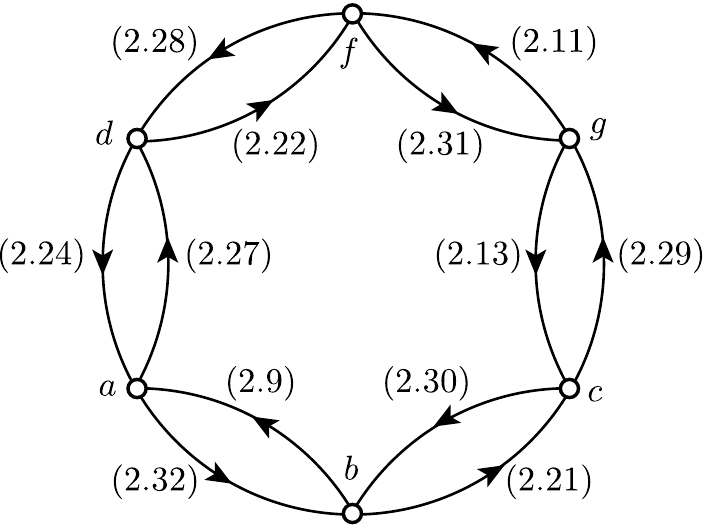}}
\caption{Twelve atoms of $\Quo A$}\label{figbfcircle}
\end{figure}

In addition to the fact that the Disjoint Paths Principle (see Observation~\ref{obs:DPPsTrs}) has played  an important role in the proof above, this principle is also useful to simplify the proof of the following lemma. This lemma  is implicit in Kulin~\cite{KulinQ5}; see the proof of Part (i) of Theorem 2.1 there.

\begin{lemma}[Kulin \cite{KulinQ5}]\label{lemma:KEndWzQ} If $A$ is a set consisting of at least three elements and $\rho$ belongs to $\Quo A\setminus \Equ A$, then $\Equ A\cup \set\rho$ generates the lattice $\Quo A$.
\end{lemma}

Since $\Equ{\set{a,b}}\cup \set{q(a,b)}$, which is a three-element chain, is a proper sublattice of $\Quo{\set{a,b}}$, the stipulation that $A$ has at least three elements cannot be omitted from Lemma~\ref{lemma:KEndWzQ}. For the reader's convenience and also to demonstrate the power of the Disjoint Paths Principle, we are going to present a new
 proof of this lemma.

\begin{proof}[Proof of Lemma~\ref{lemma:KEndWzQ}] 
We can assume that $A$ consists of the vertices $a_0,a_1,\dots, a_{n-1}$, listed counterclockwise, of a regular $n$-gon such that $\tuple{a_0,a_1}\in \rho$ but $\tuple{a_1,a_0}\notin \rho$. This $n$-gon is non-degenerate since $n=|A|\geq 3$.  If $i,j\in\set{0,\dots,n-1}$ and $j\equiv i+1$
(mod $n$), then $\set{a_i,a_j}$, $\tuple{a_i,a_j}$, and $\tuple{a_j,a_i}$ are called an \emph{undirected edge}, a \emph{counterclockwise edge}, and a \emph{clockwise edge} of the $n$-gon, respectively.

Let $S$ denote the sublattice of $\Quo A$ generated by $\Equ A\cup \set\rho$. Then all the undirected edges
are in $S$, which means that $e(a_i,a_j)\in S$ for all $i,j\in\set{0,\dots,n-1}$ with $j\equiv i+1$. We say that the counterclockwise version and the clockwise version of an edge $\set{a_i,a_j}$ are in $S$ if 
$q(a_i,a_j)\in S$ and $q(a_j,a_i)\in S$, respectively. 
It follows from \eqref{eq:wzBmnGVshszTTRkTv} that if all the counterclockwise edges and all the clockwise edges of the $n$-gon are in $S$, then all the atoms of $\Quo A$ are in $S$ and so $S=\Quo A$ by \eqref{dg:eq:wbhWrnGtrSg}. 
It also follows from \eqref{eq:wzBmnGVshszTTRkTv} that
if the counterclockwise version of an (undirected) edge belongs to $S$, then the clockwise versions of all other 
edges are in $S$. Combining this fact with its counterpart in which the two directions are interchanged,
we obtain that if at least one directed edge is in $S$, then all directed edges are in $S$ and $S=\Quo A$. 
Therefore, using that  $\tuple{a_0,a_1}\in \rho$ but $\tuple{a_1,a_0}\notin \rho$ lead to $q(a_0,a_1)=e(a_0,a_1)\wedge \rho\in S$, we obtain the statement of the Lemma.
\end{proof}

\begin{corollary}\label{corol:quo3}
$\Quo 3$ is $(1+1+2)$-generated.
\end{corollary}

\begin{proof}
Since $\Equ 3 =\Equ{\set{a,b,c}}$ is generated by the set $\set{e(a,b),e(b,c),e(c,a)}$ of its atoms,
$\set{q(a,b),e(a,b),e(b,c),e(c,a)}$ is a  $(1+1+2)$-generating set of the lattice $\Quo 3=\Quo{\set{a,b,c}}$ by  Lemma~\ref{lemma:KEndWzQ}.
\end{proof}

\section{$(1+1+2)$-generating sets $\Quo n$ for $n=11$ and $n\geq 13$}\label{sect:bign} 
The basic concept in this section is the following; the notations $e(x,y)$ and $q(x,y)$ introduced in \eqref{eq:wsZlshRnGpcngRszq} are still in effect.

\begin{definition}[Z\'adori configuration]\label{def:zadoriconfig}
For $2\leq k\in\Nplu$, let $a_0$, $a_1$, \dots, $a_k$, $b_0$, $b_1$, \dots, $b_{k-1}$ be pairwise distinct elements of a finite set $B$. Let
\begin{equation}
\begin{aligned}
\alpha&=\bigvee_{i=1}^{k}e(a_{i-1},a_{i})\vee \bigvee_{i=1}^{k-1}e(b_{i-1},b_{i}),\quad\beta=\bigvee_{i=0}^{k-1}e(a_i,b_i)\cr
\gamma&=\bigvee_{i=1}^{k}e(a_{i},b_{i-1}),\quad 
\epsilon_0=e(a_0,b_0),\quad\text{and}\quad \eta=e(a_k,b_{k-1});
\end{aligned}
\label{eq:sZwzZtrJFcmSh}
\end{equation}
they are members of $\Equ B$. The system of these $2k+1$ elements and five equivalences of $B$ is called a \emph{Z\'adori configuration} of (odd) size $2k+1$ in $B$. The set 
\begin{equation}
A:=\set{a_0,\dots,a_k,b_0,\dots, b_{k-1}}
\label{eq:sSprthRszmRt}
\end{equation}
is the \emph{support} of this configuration.
\end{definition}

A Z\'adori configuration is easy to visualize; following Z\'adori's original drawing, we do this with the help of a graph in the following way. We say that a path in a graph is horizontal, is of slope $1$, and is of slope $-1$ if all of the edges constituting the path are such. 
For vertices $x$ and $y$ in the graph,
\begin{equation}
\begin{aligned}
\pair x y \in\alpha &\defiff \text{there is a horizontal path from $x$ to $y$};
\cr
\pair x y \in\beta &\defiff \text{there is a path of slope $-1$  from $x$ to $y$};
\cr
\pair x y \in\gamma &\defiff \text{there is a path of slope $1$  from $x$ to $y$};
\end{aligned}
\label{eq:hWTlszBrt}
\end{equation}
note that a path of length 0 is simultaneously of slope 1 and of slope $-1$, and it is also horizontal.
Also, note that \eqref{eq:hWTlszBrt} complies with \eqref{eq:sZwzZtrJFcmSh}.

For example, a Z\'adori configuration of size $11$ is given in Figure~\ref{fig11}; disregard the dashed curved edges for a while. Some of the horizontal edges are labeled by $\alpha$ but, to avoid crowdedness, not all. The same convention applies  for edges of slope $-1$ and $\beta$, and edges of slope $1$ and $\gamma$.

Z\'adori configurations played a decisive role in all papers that applied extensions of Z\'adori's method; see Subsection~\ref{ssect:history} for the list of these papers. 
Given a Z\'adori configuration in $B$ with support set $A$, see \eqref{eq:sZwzZtrJFcmSh}--\eqref{eq:sSprthRszmRt}, we define 
\begin{equation}
\Equ{\restrict B A}:=
\set{\theta\in\Equ B: \text{if } \pair x y\in\theta\text{ and }\set{x,y}\not\subseteq A,\text{ then }x=y},
\label{eq:wMwtGsRnBZt}
\end{equation}
where $x/\theta$ is the $\theta$-block of $x$. In Z\'adori~\cite{zadori}, this configuration and the following lemma assumed that $B=A$. However, this assumption is not a real restriction since the map
\begin{equation}
\Equ{\restrict B A}\to \Equ A\text{ defined by }
\theta\mapsto \theta\cap(A\times A)
\label{eq:QGptFlLcSzpDgR}
\end{equation}
is clearly an isomorphism, whereby the validity of the following lemma follows from its original particular case $B=A$.

\begin{lemma}[Z\'adori~\cite{zadori}]\label{lemma:ZdrlMM}
Assume that a Z\'adori configuration of size $2k+1$ with support $A$ is given in $B$; see \eqref{def:zadoriconfig} and \eqref{eq:sSprthRszmRt}. Then $\set{\alpha,\beta,\gamma,\epsilon_0,\eta_k}$ generates $
\Equ{\restrict B A}$.
\end{lemma}

Note that this lemma is explicitly stated in  Cz\'edli~\cite{czgdirpower} and Cz\'edli and Kulin~\cite{CzGKJconcise},
and implicitly proved (hidden in long proofs) in  Cz\'edli~\cite{czgAtEqSmall}, \cite{czgFgenLargeASM}, \cite{czgEEKgenEqu}, and \cite{czgFGQL2017}, and Cz\'edli and Oluoch~\cite{czgoluoch}.
Now we are in the position to state the result of the present section.

\begin{theorem}\label{thm:sBbszmK} 
Let $n\in\Nplu$ be a natural number.
\begin{enumeratei}
\item\label{thm:sBbszmKa} If $n\geq 11$ and $n$ is odd, then $\Quo n$ is $(1+1+2)$-generated.
\item \label{thm:sBbszmKb} If $n\geq 13$,  then $\Quo n$ is $(1+1+2)$-generated.
\end{enumeratei}
\end{theorem}

\begin{figure}[ht]
\centerline
{\includegraphics[scale=1]{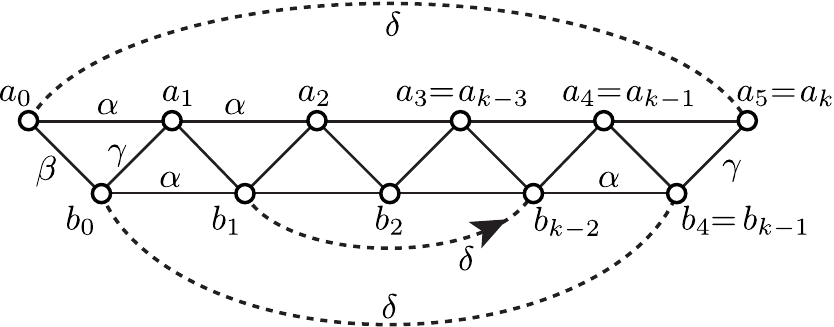}}
\caption{$\set{\alpha, \beta, \gamma, \delta}$ is a $(1+1+2)$-generating set of $\Quo {11}$}\label{fig11}
\end{figure}

To avoid complicated wording, the scope of part \eqref{thm:sBbszmKa} and that of  part \eqref{thm:sBbszmKb} are not disjoint.

\begin{proof}[Proof of Theorem \ref{thm:sBbszmK}]
First, to prove part \eqref{thm:sBbszmKa}, assume that $n=2k+1\geq 11$. Take a  Z\'adori configuration of size $2k+1$ as described in Definition~\ref{def:zadoriconfig}. Similarly to \eqref{eq:wMwtGsRnBZt}, we let
\begin{equation}
\Quo{\restrict B A}:=\set{\rho\in\Quo B: \text{ if }\pair x y\in\rho\text{ and }
\set{x,y}\not\subseteq A,\text{ then }x=y}.
\label{eq:wMwtKlrSWpMnfJ}
\end{equation}
Since the map 
\begin{equation}\Quo{\restrict B A}\to \Quo A\text{ defined by }
\rho\mapsto \rho\cap(A\times A)
\label{eq:QrvlZmsTrRf}
\end{equation}
is an isomorphism and $|A|=2k+1=n$, it suffices to show that 
$\Quo{\restrict B A}$ is $(1+1+2)$-generated.
(The advantage of not assuming that $B=A$ and working in 
$\Quo{\restrict B A}$ rather than in $\Quo A$ will be clear later in the proof of part  \eqref{thm:sBbszmKb} of the theorem.)
Define the following members of $\Quo{\restrict B A}$:
\begin{equation}
\begin{aligned}
\delta_\ast&:= e(a_0,a_k)\vee e(b_0,b_{k-1}),\quad
\delta:= \delta_\ast \vee  q(b_1,b_{k-2}),\text{ and}\cr
\delta^+&:= \delta_\ast \vee  e(b_1,b_{k-2})= \delta\vee  q(b_{k-2},b_1).
\end{aligned}
\label{eq:wGkRkKLjTlVhn}
\end{equation}
The quasiorders $\delta_\ast$ and $\delta^+$ are equivalences but $\delta$ is not. For $n=11$,  $\delta$ is visualized by dashed curved edges in Figure~\ref{fig11}. In addition to \eqref{eq:hWTlszBrt} and complying with \eqref{eq:wGkRkKLjTlVhn},  our convention for $\delta$ in  Figure~\ref{fig11} (and later in  Figure~\ref{fig14}) is that
\begin{equation}
\parbox{6cm}{$\pair x y\in\delta$ $\defiff$ there is a directed path of curved dashed edges from $x$ to $y$;}
\label{pbx:MzwkrdTss}
\end{equation}
the edges without arrow are directed in both ways.
Again, paths of length zero are permitted. By the peculiarities of $\delta$, the path in \eqref{pbx:MzwkrdTss} has to be of length 1 or 0.

Letting $S$ denote the sublattice generated by $\set{\alpha,\beta,\gamma,\delta}$ in $\Quo {\restrict B A}$, we are going to show that $S=\Quo {\restrict B A}$. Observe that
\begin{equation}
\parbox{7.3cm}{the blocks of $(\delta^+ \vee \gamma)$ are 
$\set{a_0,a_{1},a_k, b_0,b_{k-1}}$,  
$\set{a_2,a_{k-1},b_1,b_{k-2}}$, and the two-element sets  $\set{a_i,b_{i-1}}$ such that $3\leq i\leq k-2$.}
\label{pbx:wZsmnsZvGknHkGyTgl}
\end{equation}%
Hence, 
we obtain that $\beta\wedge(\delta^+ \vee \gamma)= e(a_0,b_0)$. Using that the lattice operations are isotone and 
$\pair {a_0}{b_0}$ belongs to the equivalence $\beta\wedge(\delta_\ast \vee \gamma)$, it follows from
\begin{equation}
e(a_0,b_0)\leq \beta\wedge(\delta_\ast \vee \gamma)
\leq \beta\wedge(\delta \vee \gamma)
\leq \beta\wedge(\delta^+ \vee \gamma)= e(a_0,b_0)
\label{eq:rNrNJGrTLhLBl}
\end{equation}
that $\epsilon_0:=e(a_0,b_0)= \beta\wedge(\delta \vee \gamma)\in S$. 

If we disregard $\delta$ (but keep $\delta_\ast$ and $\delta^+$), then $\beta$ and $\gamma$ play a symmetric role. This symmetry corresponds to the symmetry of Figure~\ref{fig11} across a vertical axis, if the arrow is disregarded.
Hence, $\gamma\wedge(\delta^+ \vee \beta)= e(a_k,b_{k-1})$ and $e(a_k,b_{k-1})\leq \gamma\wedge(\delta_\ast \vee \delta)$. Thus, 
\begin{equation}
e(a_k,b_{k-1})\leq \gamma\wedge(\delta_\ast \vee \beta)
\leq \gamma\wedge(\delta \vee \beta)
\leq \gamma\wedge(\delta^+ \vee \beta)= e(a_k,b_{k-1}).
\label{eq:hXsvrJkRsKkKksRv}
\end{equation}
This implies that 
\begin{equation}
\parbox{8cm}{
$\eta_k:=e(a_k,b_{k-1})=\gamma\wedge(\delta\vee \beta)\in S$.
}
\end{equation}
Therefore, $\Equ{\restrict B A}\subseteq S$ by 
 Lemma~\ref{lemma:ZdrlMM}. Since the restriction of the isomorphism given in \eqref{eq:QrvlZmsTrRf} to $\Equ{\restrict B A}$
is the isomorphism given in \eqref{eq:QGptFlLcSzpDgR}, it follows from Lemma~\ref{lemma:KEndWzQ},  $\Equ{\restrict B A}\subseteq S$, and $\delta\in \Quo{\restrict B A}\setminus \Equ {\restrict B A}$ that $S=\Quo{\restrict B A}$. Furthermore, $\delta<\alpha$ and $\set{\alpha,\beta,\gamma,\delta}$
is a $(1+1+2)$-subset of $\Quo{\restrict B A}$. Thus,
\begin{equation}
\text{ $\set{\alpha,\beta,\gamma,\delta}$ is a $(1+1+2)$-generating set of $\Quo{\restrict B A}$.}
\label{eq:ztpZwrSplNkQdG}
\end{equation}
Finally, using that we know from the sentence containing \eqref{eq:QrvlZmsTrRf} that  $\Quo{\restrict B A}\cong \Quo A$, or letting $B:=A$ when $\Quo{\restrict B A} = \Quo A$, part \eqref{thm:sBbszmKa} of the theorem follows from \eqref{eq:ztpZwrSplNkQdG}.

\begin{figure}[ht]
\centerline
{\includegraphics[scale=1]{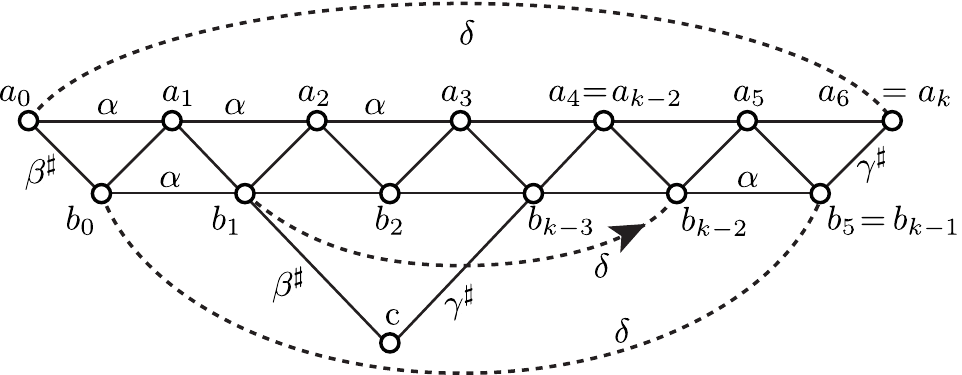}}
\caption{$\set{\alpha, \beta^\sharp, \gamma^\sharp, \delta}$ is a $(1+1+2)$-generating set of $\Quo {14}$}\label{fig14}
\end{figure}

Next, we turn our attention to  part \eqref{thm:sBbszmKb}. 
For an \emph{odd} number $n$, the validity of part \eqref{thm:sBbszmKb} follows from  part \eqref{thm:sBbszmKa}. Therefore, to  prove part \eqref{thm:sBbszmKb}, we can assume that $n\geq 14$ is an \emph{even} number. We let $k:=(n-2)/2\geq 6$. With this $k$, we use the same  Z\'adori configuration as  in part \eqref{thm:sBbszmKa} but now 
we specify that $B=A\cup\set{c}$, where $c$ is a new element outside $A$. So $|B|=|A|+1=2k+2=n$.
We still need  $\alpha$, $\beta$,  $\gamma$, $\delta_\ast$, $\delta$, $\delta^+ \in \Quo{\restrict B A}$
defined in  \eqref{eq:sZwzZtrJFcmSh} and \eqref{eq:wGkRkKLjTlVhn}.  Furthermore, we define the following two  members of $\Quo{B}$:
\begin{equation}
\beta^\sharp:=\beta\vee e(b_1,c)\quad\text{and}\quad 
\gamma^\sharp:=\gamma\vee e(b_{k-3},c).
\label{eq:wMnHwmNskRlWbT}
\end{equation}
The assumption $k\geq 6$ guarantees that \eqref{eq:wMnHwmNskRlWbT} makes sense. As opposed to the previously defined quasiorders of $\Quo{\restrict B A}$, now $\beta^\sharp$ and $\gamma^\sharp$ are not in $\Quo{\restrict B A}$. 
Note that the ``distance'' 
$(k-2)-1$ between the members of the two-element $\delta^+$-block $\set{b_1,b_{k-2}}$ as well as that between the ``suspension points'' $b_1$ and $b_{k-3}$ 
of $c$  are at least 2 and the  $\delta$-block $b_{k-3}/\delta$ is a singleton; this is why we had to assume that  $n\geq 14$, that is, $k\geq 6$. 
For the smallest value, $n=14$, the situation is visualized by  Figure~\ref{fig14}, where the conventions formulated in \eqref{eq:hWTlszBrt} and \eqref{pbx:MzwkrdTss} are valid for  $\tuple{\alpha,\beta^\sharp, \gamma^\sharp,\delta}$ instead of $\tuple{\alpha,\beta, \gamma,\delta}$.

Let $S$ be the sublattice generated by the $(1+1+2)$-subset 
$\set{\alpha,\beta^\sharp,\gamma^\sharp,\delta}$ of $\Quo B$.
We are going to show that $S=\Quo B$. Similarly to 
\eqref{pbx:wZsmnsZvGknHkGyTgl}, we observe that 
\begin{equation}
\parbox{7.5cm}{the blocks of  $(\delta^+ \vee \gamma^\sharp)$ are 
$\set{a_0,a_{1},a_k, b_0,b_{k-1}}$,  
$\set{a_2,a_{k-1},b_1,b_{k-2}}$,
$\set{a_{k-2}, b_{k-3},c}$, 
and the two-element sets  $\set{a_i,b_{i-1}}$ such that $3\leq i\leq k-3$.}
\label{pbx:szMdrnwCvdkZtpBwmR} 
\end{equation}
Hence, similarly to the  three  sentences containing  \eqref{pbx:wZsmnsZvGknHkGyTgl} and  \eqref{eq:rNrNJGrTLhLBl}, we have that $\beta^\sharp\wedge(\delta^+ \vee \gamma^\sharp)= e(a_0,b_0)$. Thus, 
\begin{equation}
e(a_0,b_0)\leq \beta^\sharp\wedge(\delta_\ast \vee \gamma^\sharp)
\leq \beta^\sharp\wedge(\delta \vee \gamma^\sharp)
\leq \beta^\sharp\wedge(\delta^+ \vee \gamma^\sharp)= e(a_0,b_0),
\end{equation}
implying that  $\epsilon_0:=e(a_0,b_0)= \beta^\sharp\wedge(\delta \vee \gamma^\sharp)\in S$.

Next, to get rid of the non-symmetrically positioned element $c$, see Figure~\ref{fig14}, observe that $\beta=(\epsilon_0\vee \alpha)\wedge \beta^\sharp\in S$ and $\gamma=(\epsilon_0\vee \alpha)\wedge \gamma^\sharp\in S$. Thus,
$\set{\alpha,\beta,\gamma,\delta}\subseteq S$, and it follows from \eqref{eq:ztpZwrSplNkQdG} that 
\begin{equation}
\Quo{\restrict B A} \subseteq S.
\label{eq:wcSlrhRmTrShmBt}
\end{equation}
In particular,  $e(b_1,b_{k-3})\in S$. Hence, using that 
the  only $(e(b_1,b_{k-3})\vee \gamma)$-block
that is not a $\gamma$-block is 
  $\set{a_2,a_{k-2}, b_1,b_{k-3},c}$, we obtain that 
$e(b_1,c)=\beta\wedge (e(b_1,b_{k-3})\vee \gamma)\in S$.
Similarly, using that $ e(b_1,b_{k-3})\in S$ by \eqref{eq:wcSlrhRmTrShmBt} and the only $(e(b_1,b_{k-3})\vee \beta)$-block that is not a $\beta$-block is $\set{a_1,a_{k-3}, b_1,b_{k-3},c}$, we obtain that $e(b_{k-3},c)=\gamma\wedge (e(b_1,b_{k-3})\vee \beta)\in S$. If $x\in A\setminus \set{b_1,b_{k-3}}$, then 
\begin{equation}
e(x,c)=\bigl(e(x,b_1)\vee e(b_1,c)\bigr)\wedge
\bigl(e(x,b_{k-3})\vee e(b_{k-3},c)\bigr)\in S
\label{eqZwhkVtgtKKnsz}
\end{equation}
by \eqref{eq:wcSlrhRmTrShmBt},   
$e(b_{1},c)\in S$, and $e(b_{k-3},c)\in  S$. 
We obtain  from $e(b_{1},c)\in S$, $e(b_{k-3},c)\in  S$, and  \eqref{eqZwhkVtgtKKnsz} that $e(x,c)\in S$ for all $x\in A$. 
This fact and \eqref{eq:wcSlrhRmTrShmBt} yield that 
$S$ contains all atoms of $\Equ B$. Using that each element of $\Equ B$ is the join of some atoms, see \eqref{dg:eq:wbhWrnGtrSg}, 
we obtain that $\Equ B\subseteq S$.  Finally, $\delta\in S\setminus \Equ B$ and Lemma~\ref{lemma:KEndWzQ} (applied to $B$ instead of $A$) imply that $S=\Quo B$. So $\Quo B$ is generated by its $(1+1+2)$-subset  $\set{\alpha,\beta^\sharp,\gamma^\sharp,\delta}$ and $|B|=2k+2=n$, completing the proof of Theorem~\ref{thm:sBbszmK}.
\end{proof}

\begin{remark}[Summary of $(1+1+2)$-generated quasiorder lattices]\label{rem:newandopen}
Now, at the end of this writing, the following is known on
the existence of $(1+1+2)$-generating sets of $\Quo n$ for  $n\in\Nplu$.
As new results, this paper proves that $\Quo n$ is $(1+1+2)$-generated for
\begin{equation}
n\in\set{3, 6, 11}\cup\set{14, 16, 18, 20, 22,\dots, 50, 52,54};
\label{eq:wKlTsRznGptGb}
\end{equation}
that is, for twenty-four new values of $n$; see Theorems~\ref{thm:wwQhTszNsJk} and \ref{thm:sBbszmK}   and Corollary~\ref{corol:quo3}. We know that, trivially, $\Quo 1$ and $\Quo 2$ are not $(1+1+2)$-generated. 
In addition to \eqref{eq:wKlTsRznGptGb}, 
 $\Quo n$ is $(1+1+2)$-generated for  all $n\geq 13$; see Theorem~\ref{thm:sBbszmK}. For 
\begin{equation}
n\in\set{4,5,7,8,9,10,12},
\end{equation}
we do not know whether $\Quo n$ is $(1+1+2)$-generated.
\end{remark}

\end{document}